\documentclass[12pt,a4paper]{article}
\usepackage[latin1,utf8]{inputenc}
\usepackage{amsmath}
\usepackage{amsfonts}
\usepackage{amssymb}
\usepackage{color}
\usepackage[english]{babel}
\usepackage{amsmath,amsthm}
\usepackage{amsfonts}
\usepackage{appendix}
\usepackage{enumitem}
\usepackage{hyperref}
\usepackage{natbib}

\usepackage{graphicx,pict2e}
\usepackage{rotating}
\usepackage{algorithm}
\floatname{algorithm}{Algorithm}

\usepackage{color}

\newtheorem{theo}{Theorem}
\newtheorem{cor}{Corollary}
\newtheorem{prop}{Proposition}
\newtheorem{lemma}{Lemma}

\newcommand{\subscript}[2]{$#1 _ #2$}
\numberwithin{equation}{section}

\title{Exponential inequalities for sampling designs}
\date{\today}

\author{Guillaume Chauvet\thanks{ENSAI/IRMAR, Campus de Ker Lann, 35170 Bruz, France. E-mail: chauvet@ensai.fr} \\
Mathieu Gerber\thanks{University of Bristol, Clifton, Bristol BS8 1TW, UK. E-mail: mathieu.gerber@bristol.ac.uk}}

\begin{document}

\maketitle

\begin{abstract}
In this work we introduce a general approach, based on the martingale representation of a sampling design and Azuma-Hoeffding's inequality, to derive exponential  inequalities for the difference between a Horvitz-Thompson estimator and its expectation. Applying this idea, we  derive a new exponential inequality for conditionally negatively associated (CNA) sampling designs, which is shown  to  improve over two existing inequalities that can be used in this context. We establish that  Chao's procedure, Till\'e's elimination procedure and the generalized Midzuno method   are CNA sampling designs, and thus obtain an exponential inequality for these  three sampling procedures. Lastly, we show that our general approach can be useful beyond CNA sampling designs by deriving an exponential inequality for Brewer's method, for which the CNA property has not been established.
\end{abstract}

\section{Introduction} \label{sec0}

 In this paper we establish exponential inequalities for the difference between a Horvitz-Thompson estimator and its expectation under various sampling designs. The resulting bounds for the tail probabilities can be computed explicitly when the population size is known, and when the variable of interest is bounded by a known constant. Under these two conditions, the results presented below can be used in practice to compute tight confidence intervals for the quantity of interest, as well as the sample size needed to guarantee that the estimation error is not larger than some chosen tolerance level $\epsilon>0$, with probability at least equal to some chosen confidence level $1-\eta$. These inequalities are also needed to prove the consistency of estimated quantiles \citep{sha:rao:93,che:wu:02}.

An important by-product of this work is  to extend the list of sampling designs that have been proven to be negatively associated (NA, see Section \ref{sec3} for a definition). This list notably contains simple random sampling without replacement \citep{joa:pro:83}, conditional Poisson sampling and Pivotal sampling \citep{dubhashi2007positive}, as well as Rao-Sampford sampling and Pareto sampling \citep{bra:jon:12}. In this paper we show that Chao's procedure \citep{cha:82}, Till\'e's elimination procedure \citep{til:96} and the generalized Midzuno method \citep{mid:52,dev:til:98} are also NA sampling designs. Showing that a sampling procedure is NA   is particularly useful  since its statistical properties can then be readily deduced from the general theory for NA random variables. For instance, Hoeffding's inequality and the bounded difference inequality  have been proven to remain valid for NA random variables \citep{farcomeni2008some}, while a maximal inequality and a Bernstein-type inequality for NA random variables have been derived in  \citet{shao2000comparison} and \citet{ber:cle:19}, respectively.

Actually, we establish below that Chao's procedure, Till\'e's elimination procedure  and the generalized Midzuno method are not only NA, but also conditionally negatively associated (CNA, see Section \ref{sec3} for a definition), and derive a general exponential inequality for such sampling designs. As a consequence of this strong property, both a result obtained assuming equal inclusion probabilities  and some numerical experiments show that the inequality  we obtain for CNA sampling designs leads to significant  improvements compared to the bound obtained by applying the Bernstein inequality for NA random variables of \citet{ber:cle:19}.  However,  this latter is not uniformly dominated by the bound that we obtain. We also compare the inequality we derive for CNA sampling designs with the one obtained from the result in \citet{pem:14}, and show that the former is sharper than the latter.

The strategy we follow to derive our exponential inequalities is to work with the  martingale representation of a sampling design (see Section \ref{sub:martingale}) and then apply Azuma-Hoeffding's inequality. The final inequalities are finally obtained by controlling the terms appearing in the Azuma-Hoeffding's bounds.

In addition to allow the derivation of a sharp exponential inequality  for CNA  sampling designs, the strategy we follow has the merit to be applicable for sampling designs which are not NA. To the best of our knowledge, an exponential inequality for such sampling designs
  only exists for successive sampling \citep{ros:72}, as recently proved by \citet{ben2018weighted}. Using our general approach we derive an exponential inequality for Brewer's method \citep{bre:63, bre:75}, which is a very simple draw by draw procedure for the selection of a sample with any prescribed set of inclusion probabilities. Whether or not the NA property holds for Brewer's method remains an open problem.

The rest of the paper is organized as follows. In Section \ref{sec1} we introduce the set up that we will consider throughout this work, as well as the martingale representation of a sampling design and a key preliminary result (Theorem \ref{theo1}).  In Section \ref{sec3} we give the exponential inequality  for CNA sampling designs (Theorem \ref{theo2}) and establish that Chao's procedure, Till\'e's elimination procedure  and the generalized Midzuno method are CNA sampling methods (Theorem \ref{prop1}). In this section, we also compare the bound obtained for CNA sampling procedures with the one  obtained by applying the Bernstein inequality \citep{ber:cle:19} and with the one derived from  the results of \cite{pem:14}. Section \ref{sec4} contains the result for Brewer's method. We conclude in Section \ref{sec:conc}. All the proofs are gathered in the appendix.

\section{Preliminaries} \label{sec1}

\subsection{Set-up and notation}

We consider a finite population $U$ of size $N$, with a variable of interest $y$ taking the value $y_k$ for the unit $k \in U$. We suppose that a random sample $S$ is selected by means of a sampling design $p(\cdot)$, and we let $\pi_k = Pr(k \in S)$ denote the probability for unit $k$ to be selected in the sample.  We let $\pi_U=( \pi_1,\ldots,\pi_N)^{\top}$ denote the vector of inclusion probabilities, and
    \begin{eqnarray} \label{sec1:eq1}
    I_U & = & (I_1,\ldots,I_N)^{\top}
    \end{eqnarray}
denote the vector of sample membership indicators. We let $n = \sum_{k \in U} \pi_k$ denote the average sample size. Recall that $p(\cdot)$ is called a fixed-size sampling design if only the subsets $s$ of size $n$ have non-zero selection probabilities $p(s)$.

 We suppose that $\pi_k>0$ for any $k \in U$, which means that there is no coverage bias. When the units are selected with equal probabilities, we have
    \begin{eqnarray} \label{sec1:eq2}
      \pi_k & = & \frac{n}{N}.
    \end{eqnarray}
When some positive auxiliary variable $x_k$ is known at the sampling stage for any unit $k$ in the population, another possible choice is to define inclusion probabilities proportional to $x_{k}$. This leads to probability proportional to size ($\pi$-ps) sampling, with
    \begin{eqnarray} \label{sec1:eq3}
      \pi_{k} & = & n \frac{x_{k}}{\sum_{l \in U} x_{l}}.
    \end{eqnarray}
Equation (\ref{sec1:eq3}) may lead to probabilities greater than $1$ for units with large values of $x_{k}$. In such case, these probabilities are set to $1$, and the other are recomputed until all of them are lower than $1$ \citep[][Section 2.10]{til:06}.

The Horvitz-Thompson (HT) estimator is
    \begin{eqnarray} \label{sec1:eq4}
      \hat{t}_{y\pi} & = & \sum_{k \in S} \check{y}_k
    \end{eqnarray}
where $\check{y}_k=y_k/\pi_k$. The HT-estimator is design-unbiased for the total $t_y=\sum_{k \in U} y_k$, in the sense that $E_p(\hat{t}_{y\pi})=t_y$, with $E_p(\cdot)$ the expectation with respect to the sampling design.

\subsection{Martingale representation\label{sub:martingale}}

A sampling design may be implemented by several sampling algorithms. For example, with a draw by draw representation, the sample $S$ is selected in $n$ steps and each step corresponds to the selection of one unit. With a sequential representation, each of the $N$ units in the population is successively considered for sample selection, and the sample is therefore obtained in $N$ steps. In this paper, we are interested in the representation of a sampling design by means of a martingale.

We say that $p(\cdot)$ has a martingale representation \citep[][Section 3.4]{til:06} if we can write the vector of sample membership indicators as
    \begin{eqnarray*}
      I_U & = & \pi_U + \sum_{t=1}^T \delta(t),
    \end{eqnarray*}
where $\{\delta(t);~t=1,\ldots,T\}$ are martingale increments with respect to some filtration $\{\mathcal{F}_t;~t=0,\ldots,T-1\}$. This definition is similar to that in \citet[Definition 34]{til:06}, although we express it in terms of martingale increments rather than in terms of the martingale itself.

 We confine ourselves to the study of a sub-class of martingale representations, proposed by \citet{dev:til:98} and called the general splitting method. There is no loss of generality of focussing on this particular representation,
 since it can be shown that any sampling method may be represented as a particular case of the splitting method in $T=N$ steps, see Appendix \ref{app0}. The method is described in \citet[Algorithm 6.9]{til:06}, and is reminded in Algorithm \ref{algo1}. Equations \eqref{algo1:eq1} and \eqref{algo1:eq2} ensure that $\delta(t)$ is a martingale increment, and equation \eqref{algo1:eq3} ensures that at any step $t=1,\ldots,T$, the components of $\pi(t)$ remain between $0$ and $1$. Our definition of the splitting method is slightly more general than in \citet{til:06}.

\begin{algorithm}[t]
\begin{enumerate}
    \item We initialize with $\pi(0)=\pi_U$.
    \item At Step $t$, if some components of $\pi(t-1)$ are not $0$ nor $1$, proceed as follows:
        \begin{enumerate}
          \item Build a set of $M_t$ vectors $\delta^1(t),\ldots,\delta^{M_t}(t)$ 
          and a set of $M_t$ non-negative scalars $\alpha^1(t),\ldots,\alpha^{M_t}(t)$ such that
            \begin{eqnarray}
              \sum_{i=1}^{M_t} \alpha^i(t) & = & 1, \label{algo1:eq1} \\
              \sum_{i=1}^{M_t} \alpha^i(t) \delta^i(t) & = & 0, \label{algo1:eq2} \\
              0 \leq \pi(t-1)+ \delta^{i}(t) \leq 1 & & \textrm{ for all } i=1,\ldots,M_t, \label{algo1:eq3}
            \end{eqnarray}
          where the inequalities in \eqref{algo1:eq3} are interpreted component-wise.
          \item Take $\delta(t)=\delta^{i}(t)$ with probability $\alpha^i(t)$, and $\pi(t)=\pi(t-1)+\delta(t)$.
        \end{enumerate}
    \item The algorithm stops at step $T$ when all the components of $\pi(T)$ are $0$ or $1$. We take $I_U=\pi(T)$.
 \end{enumerate}
\caption{General splitting method} \label{algo1}
\end{algorithm}

 If the sampling design $p(\cdot)$ is described by means of the splitting method in Algorithm \ref{algo1}, we may rewrite
    \begin{eqnarray} \label{mar:dec}
      \hat{t}_{y\pi}-t_y = \sum_{t=1}^T \xi(t) & \textrm{ where } & \xi(t) = \sum_{k \in U(t)} \check{y}_k \delta_k(t),
    \end{eqnarray}
where $\{\xi(t);~t=1,\ldots,T\}$ are martingale increments with respect to $\{\mathcal{F}_t;~t=0,\ldots,T-1\}$, and where
    \begin{eqnarray*}
      U(t) & = & \{k \in U;~\delta_k(t) \neq 0\}
    \end{eqnarray*}
is the subset of units which are treated at Step $t$ of the splitting method. Writing $\xi(t)$ in equation (\ref{mar:dec}) in terms of a sum over $U(t)$ rather than a sum over $U$ is helpful to establish the order of magnitude of $\xi(t)$, since $U(t)$ may be much smaller than $U$ for particular sampling designs like pivotal sampling \citep{dev:til:98,cha:12} or the cube method \citep{dev:til:04}.

\subsection{A preliminary result}

The inequalities presented in the next two sections rely on Theorem \ref{theo1} below, which provides an exponential inequality for a general sampling design $p(\cdot)$. Theorem \ref{theo1} is a direct consequence of the Azuma-Hoeffding inequality, and its proof is therefore omitted.

\begin{theo} \label{theo1}
Suppose that the sampling design $p(\cdot)$ is described by the splitting method in Algorithm \ref{algo1}
, and that some constants $\{a_t(n,N);\,\, t=1,\dots,T\}$ exist such that
        \begin{eqnarray*}
       Pr\Big( \sum_{k \in U(t)} |\delta_k(t)| \leq a_t(n,N) \Big)=1,\quad t=1,\dots,T.
        \end{eqnarray*}
Then for any $\epsilon>0$,
    \begin{eqnarray} \label{theo1:eq2}
     Pr(\hat{t}_{y\pi}-t_y \geq N \epsilon) & \leq & \exp\left(-\frac{N^2 \epsilon^2}{2 \{\sup|\check{y}_k|\}^2 \sum_{t=1}^T \{a_t(n,N)\}^2} \right).
    \end{eqnarray}
\end{theo}

 We are particularly interested in sampling designs with fixed size $n$. By using a draw by draw representation \citep[][Section 3.6]{til:06}, any such sampling design may be described as a particular case of the splitting method in $T=n$ steps, see Appendix \ref{app11}. 

 Based on this observation, the exponential inequalities  derived in Sections \ref{sec3} and \ref{sec4} are obtained by showing that, for the sampling designs considered, the  quantities $a_t(n,N)$ appearing in Theorem \ref{theo1} are bounded above by a constant $C$, uniformly in $t=1,\ldots,n$. In this case, Theorem \ref{theo1} yields
\begin{equation}\label{theo1:eq3}
\begin{split}
 Pr( \hat{t}_{y\pi}-t_y \geq N \epsilon)  &\leq \exp\left(-\frac{N^2 \epsilon^2}{2 \{\sup|\check{y}_k|\}^2 n C^2} \right).
 \end{split}
\end{equation}
Since the bound  in \eqref{theo1:eq3} also holds for $ Pr(t_y-\hat{t}_{y\pi} \geq N \epsilon)$,  multiplying it by two provides an upper bound for  the tail probability $ Pr(|\hat{t}_{y\pi}-t_y| \geq N \epsilon)$.

It is worth mentioning that the bound \eqref{theo1:eq2} is not tight and can be improved using a refined version of Azuma-Hoeffding inequality, such as the one derived in \citet{sason2011refined}. The resulting bound would however have a more complicated expression, and for that reason we prefer to stick with the classical Azuma-Hoeffding inequality in this paper.

\subsection{Assumptions}

In what follows we shall consider the following assumptions:
    \begin{enumerate}[label=(\subscript{H}{{\arabic*}})]
      \item\label{H:CNA} The sampling design is of fixed size $n$. Also, for any $k \neq l \in U$, for any subset $I=\{j_1,\ldots,j_p\} \subset U \setminus \{k,l\}$ with $p \leq n-2$, we have
        \begin{eqnarray} \label{sec1:eq5}
          \pi_{kl|j_1,\ldots,j_p} & \leq & \pi_{k|j_1,\ldots,j_p} \pi_{l|j_1,\ldots,j_p},
        \end{eqnarray}
       with the notation $\pi_{\cdot|j_1,\ldots,j_p} \equiv Pr(\cdot \in S|j_1,\ldots,j_p \in S)$,
       \item\label{H:Bound}  There exists some constant $M$ such that $|y_k| \leq M$ for any $k \in U$,
      \item\label{H:pi}  There exists some constant $c>0$ such that $cN^{-1} n \leq \pi_k$ for any $k \in U$.
    \end{enumerate}

We call assumption \ref{H:CNA} the conditional Sen-Yates-Grundy conditions: with $I = \emptyset$, assumption \ref{H:CNA} implies the usual Sen-Yates-Grundy conditions. Equation  \eqref{sec1:eq5} is equivalent to:
    \begin{eqnarray} \label{sec1:eq6}
      \pi_{k|j_1,\ldots,j_p,l} & \leq & \pi_{k|j_1,\ldots,j_p} \textrm{ for any distinct units } k,l,j_1,\ldots,j_p.
    \end{eqnarray}
Equation \eqref{sec1:eq6} states that adding some unit $l$ to the units already selected always decreases the conditional probability of selection of the remaining units.

Assumption \ref{H:CNA} is linked to the property of conditional negative association,
as discussed further in Section \ref{sec3}  where we consider several sampling designs for which we prove that \ref{H:CNA} holds.

Assumptions \ref{H:Bound} and \ref{H:pi} are common in survey sampling. It is assumed in \ref{H:Bound} that the variable $y_k$ is bounded. It is assumed in \ref{H:pi} that no unit has a first-order inclusion probability of smaller order than the other units, since the mean value of inclusion probabilities is
    \begin{eqnarray*}
      \bar{\pi} & = & \frac{1}{N} \sum_{k \in U} \pi_k = \frac{n}{N}.
    \end{eqnarray*}

\section{CNA  martingale sampling designs} \label{sec3}

 A sampling design $p(\cdot)$ is said to be negatively associated (NA) if for any disjoint subsets $A,B \subset U$ and any non-decreasing function $f,g$, we have
    \begin{eqnarray} \label{sec3:eq1}
      Cov\left[f(I_k,k \in A),g(I_l,l \in B) \right] & \leq & 0.
    \end{eqnarray}
It is said to be conditionally negatively associated (CNA) if the sampling design obtained by conditioning on any subset of sample indicators is NA. Obviously, CNA implies NA.

It follows from the Feder-Mihail theorem \citep{fed:mih:92} that for any fixed-size sampling design, our Assumption \ref{H:CNA} is equivalent to the CNA property. Assumption \ref{H:CNA} therefore gives a convenient way to prove CNA.

\begin{theo} \label{theo2}
If Assumption  
\ref{H:CNA} holds, then
    \begin{eqnarray} \label{cor1:eq0}
      Pr(\hat{t}_{y\pi}-t_y \geq N \epsilon) & \leq & \exp\left(-\frac{N^2 \epsilon^2}{8 n \{\sup|\check{y}_k|\}^2 } \right),\quad\forall\epsilon\geq 0.
    \end{eqnarray}
If in addition Assumptions \ref{H:Bound}-\ref{H:pi} hold, then
    \begin{eqnarray*}
     Pr(\hat{t}_{y\pi}-t_y \geq N \epsilon) & \leq & \exp\left(-\frac{nc^2 \epsilon^2}{8 M^2} \right),\quad\forall\epsilon\geq 0.
    \end{eqnarray*}
\end{theo}

Remark that it is shown in Theorem \ref{theo2} that under the Assumption
\ref{H:CNA} the inequality \eqref{theo1:eq3} holds with $C=2$.

\subsection{Comparison with \citet{ber:cle:19}\label{sub:compare}}

Under Assumption \ref{H:CNA}  the sampling design is NA and therefore an alternative exponential inequality can be obtained from Theorems 2 and 3 in \citet{ber:cle:19}. Using these latter results, we obtain:
\begin{equation}\label{eq:Bertail}
\begin{split}
Pr&(\hat{t}_{y\pi} -t_y \geq N \epsilon)\\
 &\leq 2 \exp\bigg( -\frac{   \epsilon^2 N }{8 (1-n/N)    \sup\{y_k^2/\pi_k\}+     \epsilon (4/3)  \sup(|\check{y}_k|)}\bigg).
\end{split}
\end{equation}
It is important to note that Theorem 2 in \citet{ber:cle:19} holds for any NA sampling design, while our Theorem \ref{theo2} is limited to sampling designs with the stronger CNA assumption (implying NA).

 Providing a detailed comparison of the upper bounds in \eqref{cor1:eq0} and in \eqref{eq:Bertail} is beyond the scope of the paper. The following proposition however shows that, as one may expect, the stronger CNA condition imposed in our Theorem \ref{theo2} may lead to a sharper exponential inequality.

\begin{prop}\label{prop0}
Assume that $\pi_k=n/N$ for all $k\in U$. Then,  the upper bound in  \eqref{cor1:eq0} is  smaller than the upper bound in \eqref{eq:Bertail}
\begin{itemize}
\item for all $\epsilon> 0 $  if  $n     < \big(\log(2)(8/9)\big)^{1/3}N^{2/3}$
\item for all $\epsilon\in\Big(0, \big(3-\sqrt{2}\big)(n/N)\sup|y_k|\Big]$ if $n     \geq \log(2)(8/9)(N/n)^2$.
\end{itemize}
\end{prop}

This proposition  suggests that  Theorem \ref{theo2} improves   the inequality  \eqref{eq:Bertail} when the sample size $n$ is small or when $\epsilon$ is not too large. Remark that in case of equal probabilities, the inequalities discussed in this paper are useful only for  $\epsilon<\epsilon^*:=2(1-n/N)\sup|y_k|$, since  $Pr (\hat{t}_{y\pi} -t_y \geq N \epsilon)=0$ for all $\epsilon>\epsilon^*$.  Therefore, under the assumptions of the Proposition \ref{prop0},  if $(n/N)\geq 2/(5-\sqrt{2})\approx 0.56$ then the upper bound in \eqref{cor1:eq0} is smaller than the upper bound in \eqref{eq:Bertail}  for all relevant values of $\epsilon>0$, that is for all $\epsilon\in(0,\epsilon^*]$.

To assess the validity of the conclusions of Proposition \ref{prop0} when we have unequal inclusion probabilities $(\pi_1,\dots,\pi_N)$ we consider the numerical example proposed in  \citet{ber:cle:19}. More precisely, we let $\gamma_1,\dots,\gamma_N$ be $N=10^4$ independent draws from the exponential distribution with mean 1, $(\epsilon_1,\dots,\epsilon_N)$ be $N$ independent draws from the $\mathcal{N}(0,1)$ distribution and,  for $k\in U$, we let
$$
x_k=1+\gamma_k,\quad \pi_k=\frac{n x_k}{\sum_{l\in U}x_l},\quad y_k= x_k+\sigma\epsilon_k
$$
where the parameter  $\sigma\geq 0$ allows to control the correlation between $x_k$ and $y_k$.

Figure \ref{fig:comp} shows the difference between  the upper bound in  \eqref{eq:Bertail}
 and the upper bound in \eqref{cor1:eq0} as a function of $\epsilon$,    for $\sigma\in\{0,0.5,1,5\}$ and for $n\in\{10^2,10^{2.5},10^3,10^{3.5}\}$. The results in Figure \ref{fig:comp} confirm that the inequality  \eqref{cor1:eq0} tends to be sharper than the inequality \eqref{eq:Bertail} when $n$ is small and/or when $\epsilon$ is not too large.  It is also worth noting that, globally, the improvements of the former inequality compared to the latter increase as $\sigma$ decreases (i.e.\ as the correlation between $x_k$ and $y_k$ increases). In particular, for $\sigma=0$ the bound \eqref{cor1:eq0} is smaller than the bound \eqref{eq:Bertail} for all the considered values of $n$ and of $\epsilon$.

\begin{figure}
\includegraphics[scale=0.26, trim=1.5cm 0 0 0]{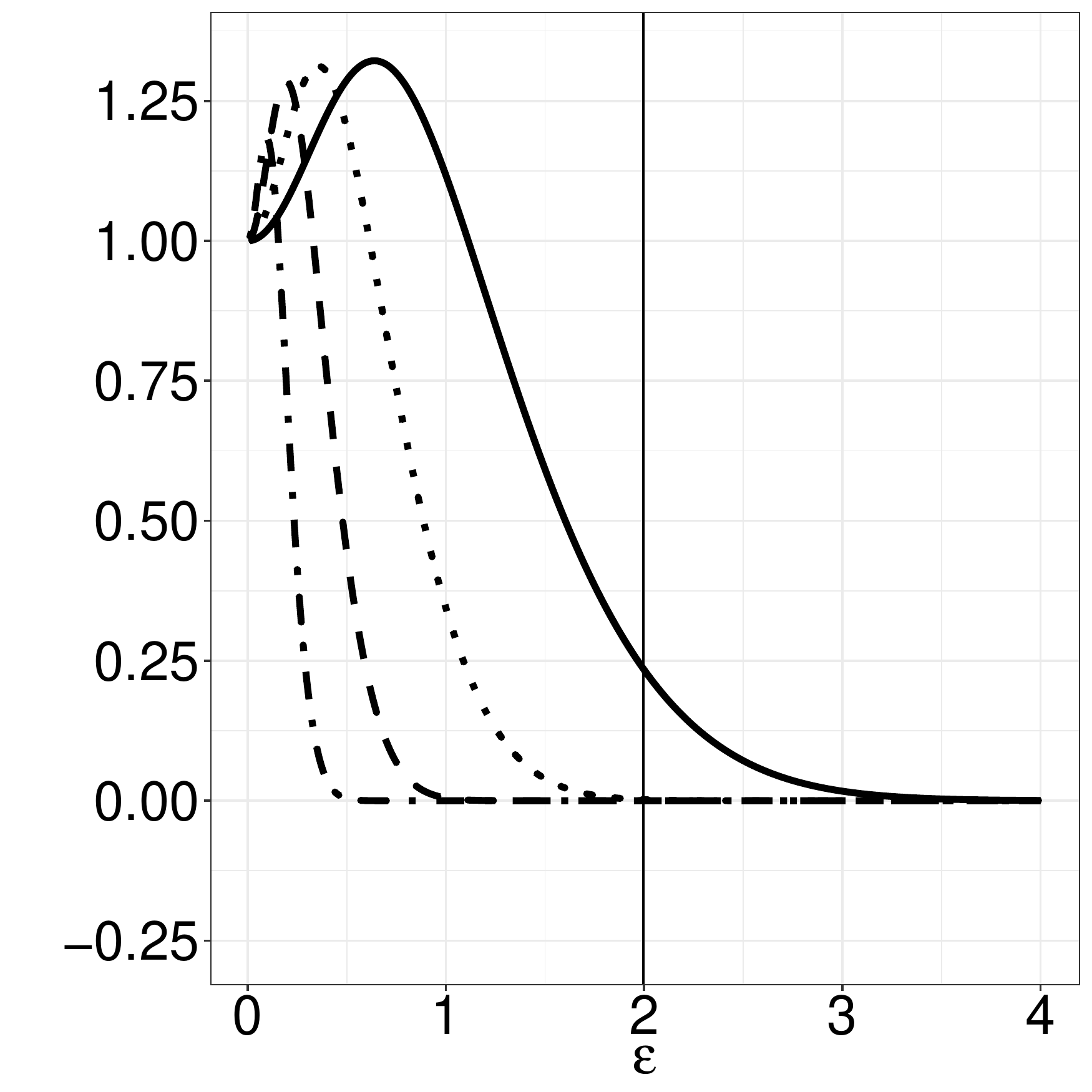}\includegraphics[scale=0.26]{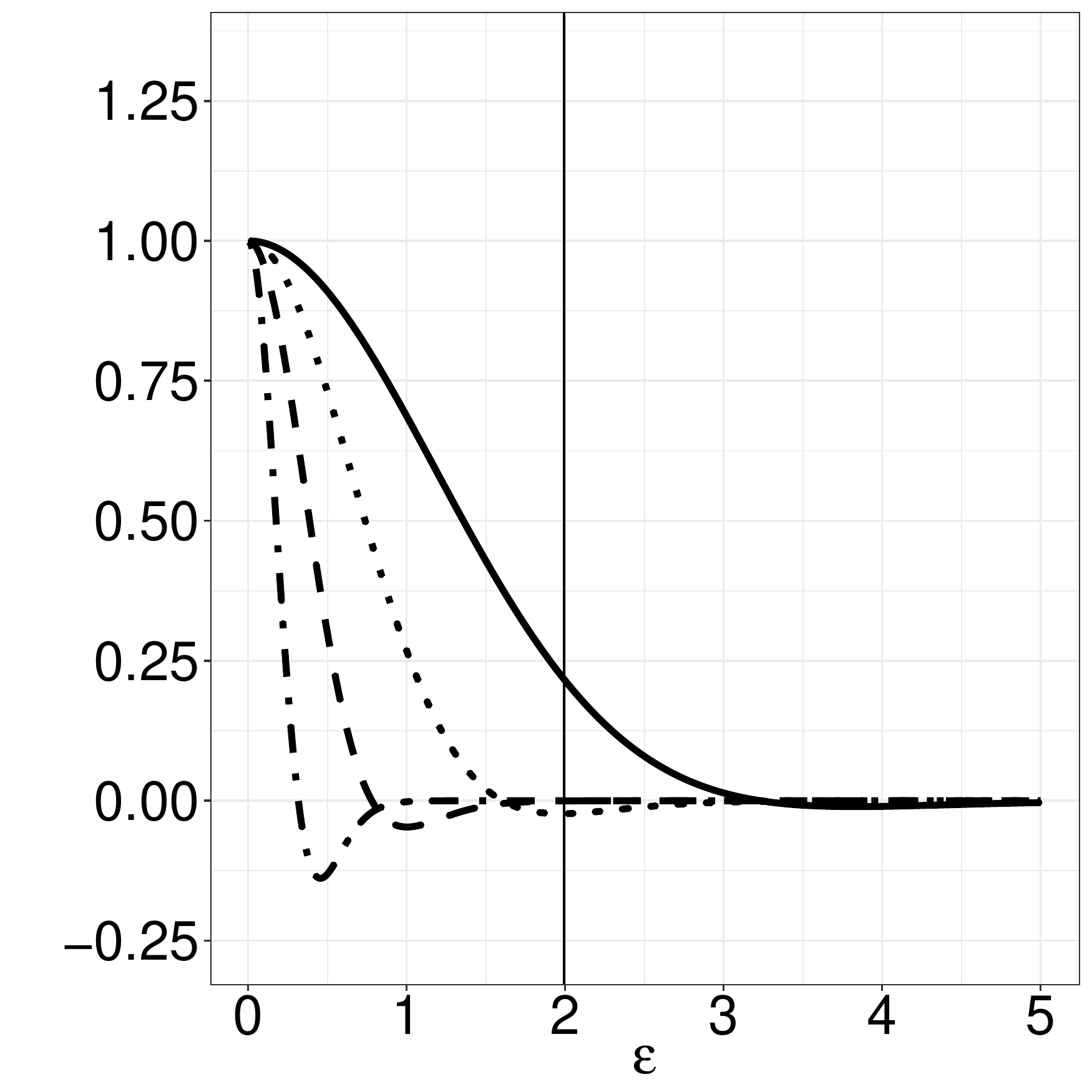}\includegraphics[scale=0.26]{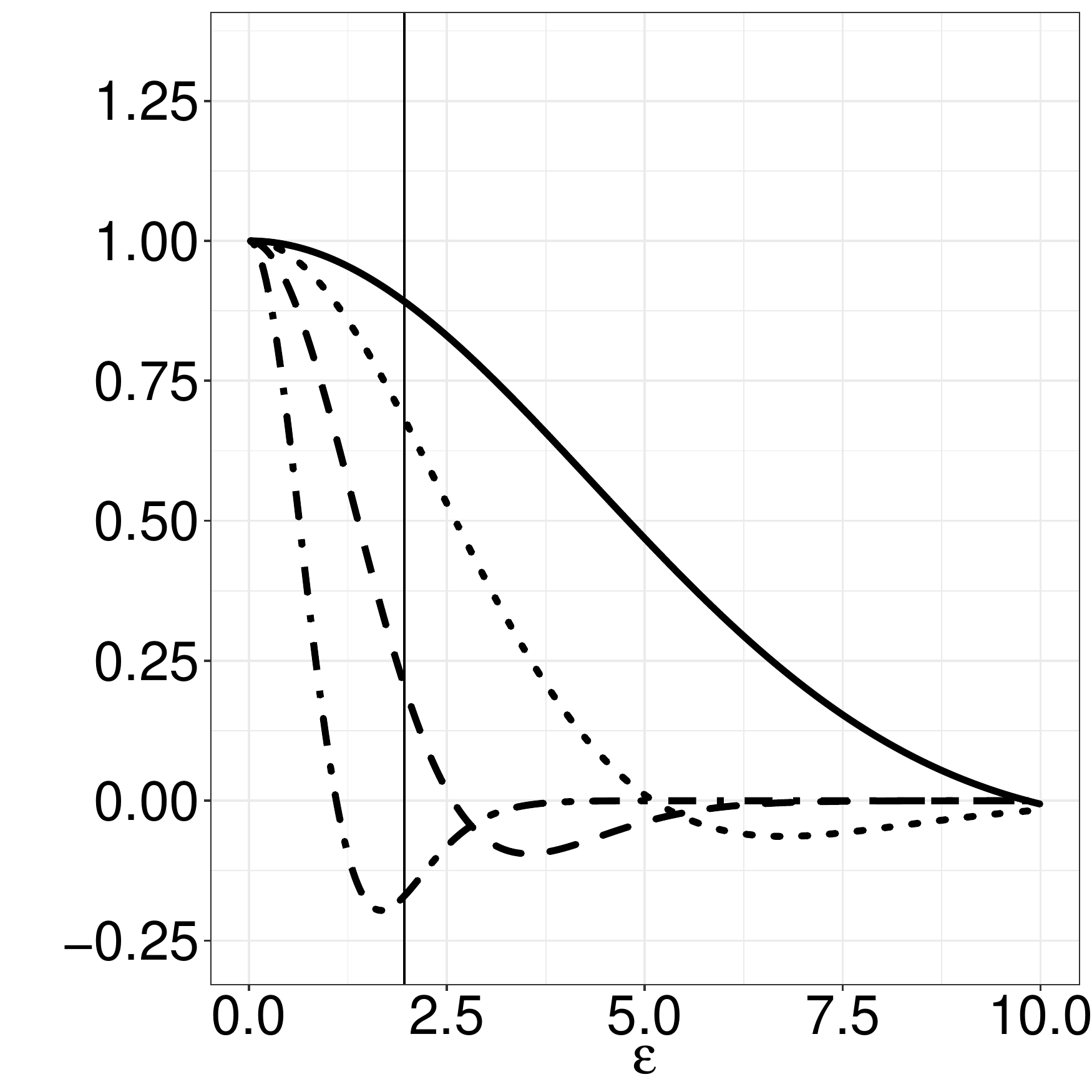}
\caption{Difference between  the upper bound  in  \eqref{eq:Bertail}
 and the upper bound in \eqref{cor1:eq0} as a function of $\epsilon$ and for the example of Section \ref{sub:compare}. Results are for $\sigma=0$ (left plot), $\sigma=1$ (middle plot) and $\sigma=5$ and are obtained for $n=10^2$ (black lines), $n=10^{2.5}$ (dotted lines), $n=10^3$ (dashed lines) and for $n=10^{3.5}$. The vertical lines show the population mean $N^{-1}\sum_{k\in Y}y_k$.\label{fig:comp}}
\end{figure}

\subsection{Comparison with \citet{pem:14}\label{sub:compare2}}
Under Assumption \ref{H:CNA} another exponential inequality can be obtained by using Proposition  2.1 and Theorem  3.4 in \citet{pem:14}, which state that for a 1-Lipschitz function  $f:\{0,1\}^N\rightarrow\mathbb{R}$ we have
\begin{align}\label{eq:pem}
Pr\big(f(I_U)-\mathbb{E}[I_U]\geq \epsilon)\leq \exp\Big(-\frac{\epsilon^2}{8 n}\Big),\quad\forall \epsilon\geq 0.
\end{align}
To use this inequality in our context let $f_{\mathrm{HT}}:\{0,1\}^N\rightarrow\mathbb{R}$ be defined by
$$
f_{\mathrm{HT}}(z)=\Big(\sum_{k\in U}\check{y}_k^2\Big)^{-\frac{1}{2}}\sum_{k\in U} z_k\check{y}_k,\quad z\in\{0,1\}^N
$$
and note that, by the Cauchy-Schwartz inequality,   this function is 1-Lipschitz. Therefore, under   Assumption \ref{H:CNA} and using (\ref{eq:pem}), for all $\epsilon\geq 0$ we have
\begin{equation}\label{eq:pem2}
\begin{split}
 Pr(\hat{t}_{y\pi}-t_y \geq N \epsilon)&= Pr\bigg(f_{\mathrm{HT}}(I_U)-\mathbb{E}[f_{\mathrm{HT}}(I_U)] \geq \Big(\sum_{k\in U}\check{y}_k^2\Big)^{-\frac{1}{2}} N \epsilon\bigg)\\
 &\leq \exp\Big(-\frac{N^2\epsilon^2}{8 n \sum_{k\in U}\check{y}_k^2}\Big).
\end{split}
\end{equation}
Then, since $(\sup|\check{y}_k|)^2\leq  \sum_{k\in U}\check{y}_k^2 $, we conclude that the upper bound in \eqref{cor1:eq0} is never larger than the upper bound in \eqref{eq:pem2}, and that the two bounds are equal if and only if $\check{y}_k\neq 0$ for only one $k\in U$. Notice that the result of Theorem \ref{theo2} allows to replace, in \eqref{eq:pem2}, the Euclidean norm $\|\check{y}\|_2$ of the vector $\check{y}=(\check{y}_k,\,k\in U)$ by its maximum norm $\|\check{y}\|_\infty$, where we recall that $\|\check{y}\|_\infty\leq \|\check{y}\|_2\leq \sqrt{N}\|\check{y}\|_\infty$.

\subsection{Applications of Theorem \ref{theo2} \label{sec32}}

In this subsection, we consider Chao's procedure \citep{cha:82}, Till\'e's elimination procedure \citep{til:96} and the generalized Midzuno method \citep{mid:52,dev:til:98}, for which we show that Assumption \ref{H:CNA} is fulfilled (and hence that these sampling designs are CNA). We suppose that the inclusion probabilities $\pi_k$ are defined proportionally to some auxiliary variable $x_k >0$, known for any unit $k \in U$, as defined in equation \eqref{sec1:eq3}.

Chao's procedure \citep{cha:82} is particularly interesting if we wish to select a sample in a data stream, without having in advance a comprehensive list of the units in the population. The procedure is described in Algorithm \ref{chao:samp}, and belongs to the so-called family of reservoir procedures. A reservoir of size $n$ is maintained, and at any step of the algorithm the next unit is considered for possible selection. If the unit is selected, one unit is removed from the reservoir. The presentation in Algorithm \ref{chao:samp} is due to \cite{til:06}, and is somewhat simpler than the original algorithm.

\begin{algorithm}[htb!]
\begin{itemize}
\item Initialize with $t=n$, $\pi_{k}(n)=1$ for $k=1,\ldots,n$, and $S(n)=\{1,\ldots,n\}$.
\item For $t=n+1,\ldots,N$:
    \begin{itemize}
    \item Compute the inclusion probabilities proportional to $x_{k}$ in the population $U(t)=\{1,\ldots,t\}$, namely:
        \begin{eqnarray*}
          \pi_k(t) & = & n \frac{x_{k}}{\sum_{k=1}^t x_{l}}.
        \end{eqnarray*}
    If some probabilities exceed $1$, they are set to $1$ and the other inclusion probabilities are recomputed until all the probabilities are lower than $1$.
    \item Generate a random number $u_t$ according to a uniform distribution.
    \item If $u_t \leq \pi_{t}(t)$, remove one unit ($k$, say) from $S(t-1)$ with probabilities
        \begin{eqnarray*}
          p_{k}(t) & = & \frac{1}{\pi_{t}(t)} \left\{1-\frac{\pi_{k}(t)}{\pi_{k}(t-1)}\right\} \textrm{ for } k \in S(t-1).
        \end{eqnarray*}
    Take $S(t)=S(t-1) \cup \{t\} \setminus \{k\}$.
    \item Otherwise, take $S(t)=S(t-1)$.
    \end{itemize}
\end{itemize}
\caption{Chao's procedure} \label{chao:samp}
\end{algorithm}

 Till\'e's elimination procedure \citep{til:96} is described in Algorithm \ref{chao:samp}. This is a backward sampling algorithm proceeding into $N-n$ steps, and at each step one unit is eliminated from the population. The $n$ units remaining after Step $N-n$ constitute the final sample.

\begin{algorithm}[htb!]
\begin{itemize}
\item For $i=n,\ldots,N$, compute the probabilities
        \begin{eqnarray*}
          \pi_k(i) & = & i \frac{x_{k}}{\sum_{l \in U} x_{l}}
        \end{eqnarray*}
for any $k \in U$. If some probabilities exceed $1$, they are set to $1$ and the other inclusion probabilities are recomputed until all the probabilities are lower than $1$.
\item For $t=N-1,\ldots,n$, eliminate a unit $k$ from the population $U$ with probability
        \begin{eqnarray*}
          r_{k,i} & = & 1-\frac{\pi_k(i)}{\pi_k(i+1)}.
        \end{eqnarray*}
\end{itemize}
\caption{Till\'e's elimination procedure} \label{tille:samp}
\end{algorithm}

The Midzuno method \citep{mid:52} is a unequal probability sampling design which enables to estimate a ratio unbiasedly. Unfortunately, the algorithm can only be applied if the inclusion probabilities are such that
    \begin{eqnarray*}
      \pi_k & \geq & \frac{n-1}{N-1},
    \end{eqnarray*}
which is very stringent. The algorithm is generalized in \cite{dev:til:98} for an arbitrary set of inclusion probabilities, see Algorithm \ref{mid:samp}.

\begin{algorithm}[htb!]
\begin{itemize}
\item For $i=N-n,\ldots,N$, compute the probabilities
        \begin{eqnarray*}
          \pi_k(i) & = & i \frac{(1-\pi_{k})}{\sum_{l \in U} (1-\pi_{l})}
        \end{eqnarray*}
for any $k \in U$. If some probabilities exceed $1$, they are set to $1$ and the other inclusion probabilities are recomputed until all the probabilities are lower than $1$.
\item For $t=N-1,\ldots,N-n$, select a unit $k$ from the population $U$ with probability
        \begin{eqnarray*}
          p_{k,i} & = & 1-\frac{\pi_k(i)}{\pi_k(i+1)}.
        \end{eqnarray*}
\end{itemize}
\caption{Generalized Midzuno method} \label{mid:samp}
\end{algorithm}

\begin{theo} \label{prop1}
The conditional Sen-Yates-Grundy condition  \ref{H:CNA} is respected for Chao's procedure, Till\'e's elimination procedure and the Ge\-ne\-ra\-li\-zed Mid\-zu\-no me\-thod.
\end{theo}


By  combining Theorems \ref{theo2} and \ref{prop1} we readily obtain the following result.
\begin{cor} \label{cor2}
Suppose that $p(\cdot)$ is Chao's procedure, Till\'e's elimination procedure or the Generalized Midzuno method. Then, the conclusions of Theorem \ref{theo2} hold.

\end{cor}

\section{Brewer's method} \label{sec4}

 Brewer's method is a simple draw by draw procedure for unequal probability sampling, which can be applied with any set $\pi_U$ of inclusion probabilities which sums to an integer. It was first proposed for a sample of size $n=2$ \citep{bre:63}, and later generalized for any sample size \citep{bre:75}. It is presented in Algorithm \ref{bre:samp} as a particular case of the splitting method.

\begin{algorithm}[htb!]
\begin{enumerate}
    \item At Step $1$, we initialize with $U(1)=U$ and $M_1=N$.
        \begin{enumerate}
          \item We take
            \begin{eqnarray*}
              \alpha^k(1) & = & \frac{\frac{\pi_k(n-\pi_k)}{1-\pi_k}}{\sum_{l \in U(1)} \frac{\pi_l(n-\pi_l)}{1-\pi_l}} \textrm{ for any } k \in U(1).
            \end{eqnarray*}
          \item We draw the first unit $J_1$ with probabilities $\alpha^k(1)$ for $k \in U(1)$. The vector $\pi(1)$ is such that
            \begin{eqnarray*}
              \pi_k(1) & = & \left\{\begin{array}{ll}
                                      1 & \textrm{if } k=J_1, \\
                                      \frac{(n-1)\pi_k}{n-\pi_{J_1}} & \textrm{otherwise}.
                                    \end{array}
               \right.
            \end{eqnarray*}
        \end{enumerate}
    \item At Step $t=2,\ldots,n$, we take $U(t)=U \setminus\{J_1,\ldots,J_{t-1}\}$ and $M_t=N-t+1$.
        \begin{enumerate}
          \item We take
            \begin{eqnarray*}
              \alpha^k(t) & = & \frac{\frac{\pi_k(t-1)\{n-t+1-\pi_k(t-1)\}}{1-\pi_k(t-1)}}{\sum_{l \in U(t)} \frac{\pi_l(t-1)\{n-t+1-\pi_l(t-1)\}}{1-\pi_l(t-1)}} \textrm{ for any } k \in U(t).
            \end{eqnarray*}
          \item We draw the $t$-th unit $J_t$ with probabilities $\alpha^k(t)$ for $k \in U(t)$. The vector $\pi(t)$ is such that
            \begin{eqnarray*}
              \pi_k(t) & = & \left\{\begin{array}{ll}
                                      1 & \textrm{if } k \in \{J_1,\ldots,J_t\}, \\
                                      \frac{(n-t)\pi_k(t-1)}{n-t+1-\pi_{J_t}(t-1)} & \textrm{otherwise}.
                                    \end{array}
               \right.
            \end{eqnarray*}
        \end{enumerate}
    \item The algorithm stops at step $T=n$ when all the components of $\pi(n)$ are $0$ or $1$. We take $I_U=\pi(n)$.
 \end{enumerate}
\caption{Brewer's method} \label{bre:samp}
\end{algorithm}

 This is not obvious whether Brewer's method satisfies condition \ref{H:CNA}. In particular, the inclusion probabilities of second (or superior) order have no explicit formulation, and may only be computed by means of the complete probability tree. However, as shown in the following result, the conclusions of Theorem \ref{theo2} derived for CNA sampling designs also hold for   Brewer's method.



\begin{theo} \label{theo3}
Suppose that $p(\cdot)$ is Brewer's procedure. Then
    \begin{eqnarray*}
      Pr(\hat{t}_{y\pi}-t_y \geq N \epsilon) & \leq & \exp\left(-\frac{N^2 \epsilon^2}{8 n \{\sup|\check{y}_k|\}^2 } \right),\quad\forall\epsilon\geq 0
    \end{eqnarray*}
If in addition Assumptions \ref{H:Bound}-\ref{H:pi} hold, then
    \begin{eqnarray*}
     Pr(\hat{t}_{y\pi}-t_y \geq N \epsilon) & \leq & \exp\left(-\frac{nc^2 \epsilon^2}{8 M^2} \right),\quad\forall\epsilon\geq 0.
    \end{eqnarray*}
\end{theo}
 Remark that this results shows that, for Brewer's procedure, equation \eqref{theo1:eq3} holds for $C=2$.

\section{Conclusion}\label{sec:conc}

In this paper, we have focused on fixed-size sampling designs, which may be represented by the splitting method in $T=n$ steps. Under such representation, we have shown that it is sufficient to prove that the constants $a_t(n,N)$ in Theorem \ref{theo1} are bounded above, to obtain an exponential inequality with the usual order in $n$.

Other sampling designs like the cube method \citep{dev:til:04} are more easily implemented through a sequential sampling algorithm, leading to a representation by the splitting method in $T=N$ steps. In such case, we need an upper bound of order $\sqrt{n/N}$ for the constants $a_t(n,N)$ to obtain an exponential inequality with the usual order. This is more difficult to establish. Alternatively, we may try to group the $N$ steps to obtain an alternative representation by means of the splitting method in $n$ steps, in such a way that the constants $a_t(n,N)$ are bounded above. This is an interesting matter for further research.

\bibliographystyle{apalike}

\begin{thebibliography}{}

\end{thebibliography}


\begin{thebibliography}{}

\bibitem[Ben-Hamou et~al., 2018]{ben2018weighted}
Ben-Hamou, A., Peres, Y., Salez, J., et~al. (2018).
\newblock Weighted sampling without replacement.
\newblock {\em Brazilian Journal of Probability and Statistics},
  32(3):657--669.

\bibitem[Bertail and Cl{\'e}men{\c{c}}on, 2019]{ber:cle:19}
Bertail, P. and Cl{\'e}men{\c{c}}on, S. (2019).
\newblock Bernstein-type exponential inequalities in survey sampling:
  Conditional poisson sampling schemes.
\newblock {\em Bernoulli}, 25(4B):3527--3554.

\bibitem[Br{\"a}nd{\'e}n and Jonasson, 2012]{bra:jon:12}
Br{\"a}nd{\'e}n, P. and Jonasson, J. (2012).
\newblock Negative dependence in sampling.
\newblock {\em Scand. J. Stat.}, 39(4):830--838.

\bibitem[Brewer, 1963]{bre:63}
Brewer, K.~E. (1963).
\newblock A model of systematic sampling with unequal probabilities.
\newblock {\em Australian Journal of Statistics}, 5(1):5--13.

\bibitem[Brewer, 1975]{bre:75}
Brewer, K.~E. (1975).
\newblock A simple procedure for sampling $\pi$-pswor.
\newblock {\em Australian Journal of Statistics}, 17(3):166--172.

\bibitem[Chao, 1982]{cha:82}
Chao, M. (1982).
\newblock A general purpose unequal probability sampling plan.
\newblock {\em Biometrika}, 69(3):653--656.

\bibitem[Chauvet, 2012]{cha:12}
Chauvet, G. (2012).
\newblock On a characterization of ordered pivotal sampling.
\newblock {\em Bernoulli}, 18(4):1320--1340.

\bibitem[Chen and Wu, 2002]{che:wu:02}
Chen, J. and Wu, C. (2002).
\newblock Estimation of distribution function and quantiles using the
  model-calibrated pseudo empirical likelihood method.
\newblock {\em Statistica Sinica}, pages 1223--1239.

\bibitem[Deville and Till{\'e}, 1998]{dev:til:98}
Deville, J.-C. and Till{\'e}, Y. (1998).
\newblock Unequal probability sampling without replacement through a splitting
  method.
\newblock {\em Biometrika}, 85(1):89--101.

\bibitem[Deville and Till{\'e}, 2004]{dev:til:04}
Deville, J.-C. and Till{\'e}, Y. (2004).
\newblock Efficient balanced sampling: the cube method.
\newblock {\em Biometrika}, 91(4):893--912.

\bibitem[Dubhashi et~al., 2007]{dubhashi2007positive}
Dubhashi, D., Jonasson, J., and Ranjan, D. (2007).
\newblock Positive influence and negative dependence.
\newblock {\em Combinatorics, Probability and Computing}, 16(1):29--41.

\bibitem[Esary et~al., 1967]{esa:pro:wal:67}
Esary, J.~D., Proschan, F., and Walkup, D.~W. (1967).
\newblock Association of random variables, with applications.
\newblock {\em Ann. Math. Statist.}, 38(5):1466--1474.

\bibitem[Farcomeni, 2008]{farcomeni2008some}
Farcomeni, A. (2008).
\newblock Some finite sample properties of negatively dependent random
  variables.
\newblock {\em Theory of Probability and Mathematical Statistics}, 77:155--163.

\bibitem[Feder and Mihail, 1992]{fed:mih:92}
Feder, T. and Mihail, M. (1992).
\newblock Balanced matroids.
\newblock In {\em Proceedings of the twenty-fourth annual ACM symposium on
  Theory of computing}, pages 26--38.

\bibitem[Joag-Dev et~al., 1983]{joa:pro:83}
Joag-Dev, K., Proschan, F., et~al. (1983).
\newblock Negative association of random variables with applications.
\newblock {\em The Annals of Statistics}, 11(1):286--295.

\bibitem[Midzuno, 1951]{mid:52}
Midzuno, H. (1951).
\newblock On the sampling system with probability proportional to sum of sizes.
\newblock {\em Ann. Inst. Stat. Math.}, 3:99--107.

\bibitem[Pemantle and Peres, 2014]{pem:14}
Pemantle, R. and Peres, Y. (2014)
\newblock Concentration of Lipschitz Functionals of Determinantal and Other Strong Rayleigh Measures.
\newblock{\em{Combinatorics, Probability and Computing} } 23(1):140--160.

\bibitem[Ros{\'e}n, 1972]{ros:72}
Ros{\'e}n, B. (1972).
\newblock Asymptotic theory for successive sampling with varying probabilities
  without replacement. {I}, {II}.
\newblock {\em Ann. Stat.}, 43:373--397; ibid. 43 (1972), 748--776.

\bibitem[Sason, 2011]{sason2011refined}
Sason, I. (2011).
\newblock On refined versions of the azuma-hoeffding inequality with
  applications in information theory.
\newblock {\em arXiv preprint arXiv:1111.1977}.

\bibitem[Shao and Rao, 1993]{sha:rao:93}
Shao, J. and Rao, J. (1993).
\newblock Standard errors for low income proportions estimated from stratified
  multi-stage samples.
\newblock {\em Sankhy{\=a}: The Indian Journal of Statistics, Series B}, pages
  393--414.

\bibitem[Shao, 2000]{shao2000comparison}
Shao, Q.-M. (2000).
\newblock A comparison theorem on moment inequalities between negatively
  associated and independent random variables.
\newblock {\em Journal of Theoretical Probability}, 13(2):343--356.

\bibitem[Till{\'e}, 1996]{til:96}
Till{\'e}, Y. (1996).
\newblock An elimination procedure for unequal probability sampling without
  replacement.
\newblock {\em Biometrika}, 83(1):238--241.

\bibitem[Till{\'e}, 2011]{til:06}
Till{\'e}, Y. (2011).
\newblock {\em Sampling algorithms}.
\newblock Springer.

\end{thebibliography}

\newpage

\appendix

\section{Proofs}

\subsection{A universal representation by means of the splitting method} \label{app0}

\begin{lemma} \label{lem0}
  Any sampling design $p(\cdot)$ may be represented by means of the splitting Algorithm \ref{algo1}.
\end{lemma}

\begin{proof}
A sampling design $p(\cdot)$ can always be implemented by means of a sequential procedure. At step $t=1$, the unit $1$ is selected with probability $\pi_1$, and $I_1$ is the sample membership indicator for unit $1$. At steps $t=2,\ldots,N$, the unit $t$ is selected with probability
    \begin{eqnarray*}
      Pr(t \in S | I_1,\ldots,I_{t-1}),
    \end{eqnarray*}
and $I_t$ is the sample membership indicator for unit $t$. This corresponds to the Doob martingale associated with the filtration $\sigma(I_1,\ldots,I_t)$. \\

This procedure is a particular case of the splitting Algorithm \ref{algo1}, where $T=N$; $M_t=2$ for all $t=1,\ldots,N$; $\alpha^1(t)=Pr(t \in S | I_1,\ldots,I_{t-1})$ and $\delta^1(t)$ is such that
    \begin{eqnarray*}
      \delta_l^1(t) & = & \left\{\begin{array}{ll}
                                   0 & \textrm{if } l<t, \\
                                   1-Pr(t \in S | I_1,\ldots,I_{t-1}) & \textrm{if } l=t, \\
                                   Pr(l \in S | I_1,\ldots,I_{t-1},I_t=1)-Pr(l \in S | I_1,\ldots,I_{t-1}) & \textrm{if } l >t,
                                 \end{array} \right.
    \end{eqnarray*}
and where $\alpha^2(t)=1-Pr(t \in S | I_1,\ldots,I_{t-1})$ and $\delta^2(t)$ is such that
    \begin{eqnarray*}
      \delta_l^2(t) & = & \left\{\begin{array}{ll}
                                   0 & \textrm{if } l<t, \\
                                   -Pr(t \in S | I_1,\ldots,I_{t-1}) & \textrm{if } l=t, \\
                                   Pr(l \in S | I_1,\ldots,I_{t-1},I_t=0)-Pr(l \in S | I_1,\ldots,I_{t-1}) & \textrm{if } l >t.
                                 \end{array}  \right.
    \end{eqnarray*}
\end{proof}

\subsection{Proof of Theorem \ref{theo2}} \label{app1}

\subsubsection{Preliminary results} \label{app11}

\begin{lemma} \label{lem1}
  A fixed-size sampling design $p(\cdot)$ may be obtained by means of the draw by draw sampling Algorithm \ref{algo2}.
\end{lemma}

\begin{algorithm}[htb!]
\begin{enumerate}
    \item At Step $t=1$, we initialize with $U(1)=U$ and
        \begin{eqnarray} \label{algo2:eq1}
          p_{k,1} & = & \frac{\pi_k}{n} \textrm{ for any } k \in U(1).
        \end{eqnarray}
    A first unit $J_1$ is selected in $U(1)$ with probabilities $p_{k,1}$.
    \item At Step $t>1$, we take $U(t)=U \setminus \{J_1,\ldots,J_{t-1}\}$ and
        \begin{eqnarray} \label{algo2:eq2}
          p_{k,t} & = & \frac{\pi_{k|J_1,\ldots,J_{t-1}}}{n-t+1} \textrm{ for any } k \in U(t).
        \end{eqnarray}
    A unit $J_t$ is selected in $U(t)$ with probabilities $p_{k,t}$.
    \item The algorithm stops at time $t=n$, and the sample is $S=\{J_1,\ldots,J_n\}$.
 \end{enumerate}
\caption{Draw by draw sampling algorithm for a fixed-size sampling design} \label{algo2}
\end{algorithm}

\begin{proof}
We note $\Sigma_n$ for the set of permutations of size $n$, and $\sigma$ for a particular permutation. For any subset $s=\{j_1,\ldots,j_n\} \subset U$ of size $n$, we have
    \begin{eqnarray*}
      Pr(S = s) & = & \sum_{\sigma \in \Sigma_n} Pr(J_1 = j_{\sigma(1)}, \ldots, J_n = j_{\sigma(n)}) \nonumber \\
                & = & \sum_{\sigma \in \Sigma_n} p_{j_{\sigma(1)},1} \times \cdots \times p_{j_{\sigma(n)},n} \nonumber \\
                & = & \sum_{\sigma \in \Sigma_n} \frac{\pi_{j_{\sigma(1)}}\pi_{j_{\sigma(2)}|j_{\sigma(1)}} \cdots \pi_{j_{\sigma(n)}|j_{\sigma(1)},\ldots,j_{\sigma(n-1)}}}{n!} \nonumber \\
                & = & \sum_{\sigma \in \Sigma_n} \frac{\pi_{j_{\sigma(1)},\ldots,j_{\sigma(n)}}}{n!} = \sum_{\sigma \in \Sigma_n} \frac{\pi_{j_{1},\ldots,j_{n}}}{n!} = p(s).
    \end{eqnarray*}
\end{proof}

\noindent \textbf{Remark} Algorithm \ref{algo2} is not helpful in practice to select a sample by means of the sampling design under study. This algorithm requires to determine the conditional inclusion probabilities up to any order, which are usually very difficult to compute. 

\begin{lemma} \label{lem2}
Algorithm \ref{algo2} is a particular case of Algorithm \ref{algo1} where $T=n$,  $M_t=N-t+1$ for all   $t=1,\ldots,n$,   and where, for all $t=1,\dots n$ and $i=1,\dots, M_t$,   $\alpha^i(t)=p_{i,t}$ with $p_{i,t}$ as defined in  \eqref{algo2:eq1}-\eqref{algo2:eq2} while $\delta^i(t)$ is such that
    \begin{eqnarray} \label{app1:eq1}
      \delta_l^i(t) & = & \left\{\begin{array}{ll}
                                   0 & \textrm{if } l \in \{J_1,\ldots,J_{t-1}\}, \\
                                   1-\pi_{i|J_1,\ldots,J_{t-1}} & \textrm{if } l=i, \\
                                   -(\pi_{l|J_1,\ldots,J_{t-1}}-\pi_{l|J_1,\ldots,J_{t-1},i}) & \textrm{if } l \in U(t) \setminus \{i\}.
                                 \end{array}
       \right.
    \end{eqnarray}
\end{lemma}
\begin{proof}
The lemma is a direct consequence of Lemma \ref{lem1} and of the definitions of Algorithms \ref{algo1} and \ref{algo2}.
\end{proof}

\subsubsection{Proof of the theorem} \label{app12}

\begin{proof}
 By Theorem \ref{theo1} and Lemma \ref{lem2}, to prove  Theorem \ref{theo2}, it is therefore sufficient to prove that
\begin{align}\label{eq:ToSHow}
\sum_{l \in U(t)} |\delta_l^i(t)|\leq 2,\quad\forall i\in\{1,\dots, M_t\},\quad\forall t\in\{1,\dots,n\}
\end{align}
where $M_t$  and $\big\{\{\delta_l^i\}_{i=1}^{M_t};\,t=1,\dots,n\}$ are as in Lemma \ref{lem2}.

Under Assumption \ref{H:CNA}, for any $t=1,\ldots,n$ and  y $i=1,\ldots,M_t$ we have
    \begin{eqnarray} \label{app1:eq2}
      \sum_{l \in U(t)} |\delta_l^i(t)| & = & \delta_i^i(t) + \sum_{l \in U(t) \setminus \{i\}} (\pi_{l|J_1,\ldots,J_{t-1}}-\pi_{l|J_1,\ldots,J_{t-1},i}) \\
                                        & = & \left\{1-\pi_{i|J_1,\ldots,J_{t-1}}\right\}+\left\{(n-t+1-\pi_{i|J_1,\ldots,J_{t-1}})-(n-t)\right\} \nonumber \\
                                        & = & 2 \left\{1-\pi_{i|J_1,\ldots,J_{t-1}}\right\}, \nonumber
    \end{eqnarray}
and where the second line in \eqref{app1:eq2} follows from the identities
    \begin{eqnarray*}
      \sum_{l \in U(t)} \pi_{l|J_1,\ldots,J_{t-1}}  =  n-(t-1), \quad \sum_{l \in U(t) \setminus \{i\}}  \pi_{l|J_1,\ldots,J_{t-1},i} c =  n-t.
    \end{eqnarray*}
This shows \eqref{eq:ToSHow} and the proof is complete.
\end{proof}

\subsection{Proof of Proposition \ref{prop0}}
\begin{proof}

Let $v^2=\{\sup |y_k|\}^2$ and note that
\begin{equation*}
\begin{split}
\exp\left(-\frac{ n \epsilon^2   }{8 v^2 } \right) \leq 2 \exp\bigg( -\frac{ n \epsilon^2 }{8(1-n/N)  v^2+     (4/3)\epsilon    v}\bigg)\Leftrightarrow f_n(\epsilon)\leq 0
\end{split}
\end{equation*}
 where, for every $x\geq 0$,
$$
f(x)=-\Big(\frac{4}{3} n v\Big)x^3+\Big(8n\frac{n}{N}v^2\Big) x^2-\Big(\frac{32}{3}\log(2) v^3\Big)x-64\Big(1-\frac{n}{N}\Big)v^4\log(2).
$$
A sufficient condition to have $f(\epsilon)\leq 0$ is that
\begin{align*}
 -\Big(\frac{4}{3} n v\Big)&\epsilon^3+\Big(8n\frac{n}{N}v^2\Big)\epsilon^2-\Big(\frac{32}{3}\log(2) v^3\Big)\epsilon\leq 0\Leftrightarrow g(\epsilon)\leq 0
\end{align*}
 where, for every $x\geq 0$,
$$
g(x)=-\Big(\frac{4}{3} n v\Big)x^2+\Big(8n\frac{n}{N}v^2\Big)x -\Big(\frac{32}{3}\log(2) v^3\Big).
$$
Notice that $g(0)< 0$ and that the equation has a solution $g(x)=0$ has a (real) solution  if and only if
\begin{align}\label{eq:equiv}
\Big(8n\frac{n}{N}v^2\Big)^2&-4\Big(\frac{4}{3} n v\Big)\Big(\frac{32}{3}\log(2) v^3\Big)\geq 0 \Leftrightarrow  n     \geq \log(2)\frac{8}{9}\Big(\frac{N}{n}\Big)^2.
\end{align}
This shows the first part of the proposition.

 To show the second part assume that  \eqref{eq:equiv} holds. Then, since $g(0)<0$, it follows that $g(x)\leq 0$ for all $x\in [0,x_1^*]$, where
\begin{align*}
x_1^*&=\frac{-\Big(8n\frac{n}{N}v^2\Big)+\sqrt{\Big(8n\frac{n}{N}v^2\Big)^2-4\Big(\frac{4}{3} n v\Big)\Big(\frac{32}{3}\log(2) v^3\Big)}}{2\Big(-\frac{4}{3} n v\Big)}\\
&=   3\frac{n}{N}v -\Big(  2(n/N)^2 v^2- \frac{8}{9}    \log(2) v^2/n \Big)^{1/2}\\
&\geq (3-\sqrt{2})\frac{n}{N}v.
\end{align*}
The proof is complete.
\end{proof}
\subsection{Proof of Theorem \ref{prop1}} \label{sec:pprop1}

Theorem \ref{prop1} is  a consequence of Lemmas \ref{lemma:Cha0}-\ref{lemma:Midzuno} below, which respectively show that Assumption \ref{H:CNA} holds for Chao's procedure,  Till\'e's elimination procedure and the Generalized Midzuno method.

\begin{lemma} \label{lemma:Cha0}
Assumption \ref{H:CNA} is verified for Chao's procedure.
\end{lemma}

\begin{proof}
We prove equation \eqref{sec1:eq3} by induction, using the notation
    \begin{eqnarray*}
      \pi_{\cdot|j_1,\ldots,j_p}(t) \equiv Pr\{\cdot \in S(t)|j_1,\ldots,j_p \in S(t)\}.
    \end{eqnarray*}
At step $t=n$, the equation
    \begin{eqnarray*}
    \pi_{kl|j_1,\ldots,j_p}(n) & \leq & \pi_{k|j_1,\ldots,j_p}(n) \pi_{l|j_1,\ldots,j_p}(n)
    \end{eqnarray*}
is automatically fulfilled. We now treat the case of any step $t > n$. Recall that $I=\{j_1,\ldots,j_p\}$, as defined in Assumption (H1). We need to consider three cases: (i) either $t \neq k$, $t \neq l$ and $t \notin I$; (ii) or $t=k$, $t \neq l$ and $t \notin I$; (iii) or $t \neq k$, $t \neq l$ and $t \in I$.

 We consider the case (i) first. Making use of Lemma 2 in \citet{cha:82}, we obtain
    \begin{eqnarray*}
    \pi_{k|j_1,\ldots,j_p}(t) & = & \pi_{k|j_1,\ldots,j_p}(t-1) \frac{1-\pi_t(t) \sum_{i=1}^p p_{j_i}(t)-\pi_t(t) p_{k}(t)}{1-\pi_t(t) \sum_{i=1}^p p_{j_i}(t)}, \nonumber \\
    \pi_{l|j_1,\ldots,j_p}(t) & = & \pi_{l|j_1,\ldots,j_p}(t-1) \frac{1-\pi_t(t) \sum_{i=1}^p p_{j_i}(t)-\pi_t(t) p_{l}(t)}{1-\pi_t(t) \sum_{i=1}^p p_{j_i}(t)}, \\
    \pi_{kl|j_1,\ldots,j_p}(t) & = & \pi_{kl|j_1,\ldots,j_p}(t-1) \frac{1-\pi_t(t) \sum_{i=1}^p p_{j_i}(t)-\pi_t(t) p_{k}(t)-\pi_t(t) p_{l}(t)}{1-\pi_t(t) \sum_{i=1}^p p_{j_i}(t)}. \nonumber
    \end{eqnarray*}
This leads to
    \begin{eqnarray}
      \frac{\pi_{kl|j_1,\ldots,j_p}(t)}{\pi_{k|j_1,\ldots,j_p}(t) \pi_{l|j_1,\ldots,j_p}(t)} & = &
      \frac{\pi_{kl|j_1,\ldots,j_p}(t-1)}{\pi_{k|j_1,\ldots,j_p}(t-1) \pi_{l|j_1,\ldots,j_p}(t-1)} \times \Delta_1(t), \\
      \textrm{with } \Delta_1(t) & = & \frac{\{1-\pi_t(t)(x_p+x_k+x_l)\}\{1-\pi_t(t) x_p\}}{\{1-\pi_t(t)(x_p+x_k)\}\{1-\pi_t(t)(x_p+x_l)\}}, \nonumber
    \end{eqnarray}
where we note $x_p=\sum_{i=1}^p p_{j_i}(t)$, $x_k=p_{k}(t)$ and $x_l=p_{l}(t)$, and it is easy to prove that $\Delta_1(t) \leq 1$.

 We now consider the case (ii). Making use of Lemma 2 in \citet{cha:82}, we obtain
    \begin{eqnarray*}
    \pi_{t|j_1,\ldots,j_p}(t) & = & \frac{\pi_t(t)\{1- \sum_{i=1}^p p_{j_i}(t)\}}{1-\pi_t(t) \sum_{i=1}^p p_{j_i}(t)}, \nonumber \\
    \pi_{tl|j_1,\ldots,j_p}(t) & = & \pi_{l|j_1,\ldots,j_p}(t-1) \frac{\pi_t(t)\{1- \sum_{i=1}^p p_{j_i}(t)-p_l(t)\}}{1-\pi_t(t) \sum_{i=1}^p p_{j_i}(t)}. \nonumber
    \end{eqnarray*}
This leads to
    \begin{eqnarray*}
      \frac{\pi_{kl|j_1,\ldots,j_p}(t)}{\pi_{k|j_1,\ldots,j_p}(t) \pi_{l|j_1,\ldots,j_p}(t)} & = & \Delta_2(t), \\
      \textrm{with } \Delta_2(t) & = & \frac{(1-x_p-x_l)(1-\pi_t(t) x_p)}{(1-x_p)(1-\pi_t(t)x_p-\pi_t(t)x_l)}, \nonumber
    \end{eqnarray*}
and $\Delta_2(t) \leq 1$.

Finally, we consider the case (iii). Suppose without loss of generality that $j_n=t$. Then:
    \begin{eqnarray*}
    \pi_{k|j_1,\ldots,j_{p-1},t}(t) & = & \pi_{k|j_1,\ldots,j_{p-1}}(t-1) \frac{1- \sum_{i=1}^{p-1} p_{j_i}(t)- p_{k}(t)}{1- \sum_{i=1}^{p-1} p_{j_i}(t)}, \nonumber \\
    \pi_{l|j_1,\ldots,j_{p-1},t}(t) & = & \pi_{l|j_1,\ldots,j_{p-1}}(t-1) \frac{1- \sum_{i=1}^{p-1} p_{j_i}(t)- p_{l}(t)}{1- \sum_{i=1}^{p-1} p_{j_i}(t)}, \nonumber \\
    \pi_{kl|j_1,\ldots,j_{p-1},t}(t) & = & \pi_{kl|j_1,\ldots,j_{p-1}}(t-1) \frac{1- \sum_{i=1}^{p-1} p_{j_i}(t)- p_{k}(t)- p_{l}(t)}{1- \sum_{i=1}^{p-1} p_{j_i}(t)}.
    \end{eqnarray*}
This leads to
    \begin{eqnarray*}
      \frac{\pi_{kl|j_1,\ldots,j_{p-1},t}(t)}{\pi_{k|j_1,\ldots,j_{p-1},t}(t) \pi_{l|j_1,\ldots,j_{p-1},t}(t)} & = & \Delta_3(t), \\
      \textrm{with } \Delta_3(t) & = & \frac{(1-x_{p-1}-x_k-x_l)(1-x_{p-1})}{(1-x_{p-1}-x_k)(1-x_{p-1}-x_l)}, \nonumber
    \end{eqnarray*}
where we note $x_{p-1}=\sum_{i=1}^{p-1} p_{j_i}(t)$. We have $\Delta_3(t) \leq 1$, which completes the proof.
\end{proof}

\begin{lemma} \label{lemma:Tille}
Assumption \ref{H:CNA} is verified for   Till\'e's elimination procedure.
\end{lemma}
\begin{proof}
 From Algorithm \ref{tille:samp}, we obtain
    \begin{eqnarray*}
      \frac{\pi_{kl|j_1,\ldots,j_p}}{\pi_{k|j_1,\ldots,j_p} \pi_{l|j_1,\ldots,j_p}} = \prod_{t=n}^{N-1} \frac{(1-\sum_{j=1}^p r_{j_i,t}-r_{k,t}-r_{l,t})(1-\sum_{j=1}^p r_{j_i,t})}{(1-\sum_{j=1}^p r_{j_i,t}-r_{k,t})(1-\sum_{j=1}^p r_{j_i,t}-r_{l,t})}\leq 1.
    \end{eqnarray*}
\end{proof}

\begin{lemma} \label{lemma:Midzuno}
Assumption \ref{H:CNA} is verified for the Generalized Midzuno method.
\end{lemma}

\begin{proof}  It can be shown \citep[][Section 6.3.5]{til:06} that the generalized Midzuno method is the complementary sampling design of Till\'e's elimination procedure. More precisely, if $I_U$ is generated according to the Generalized Midzuno Method with inclusion probabilities $\pi_U$, then $J_U=1-I_U$ may be seen as generated according to Till\'e's elimination procedure with inclusion probabilities $1-\pi_U$.

 The proof is therefore similar to that in \citet[][Section 4.1]{esa:pro:wal:67}. Let $A,B,C$ denote three disjoint subsets in $U$, and let $f$ and $g$ denote two non-decreasing functions. The functions
    \begin{eqnarray*}
      \bar{f}(x) = 1-f(1-x) & \textrm{ and } & \bar{g}(x) = 1-g(1-x)
    \end{eqnarray*}
are also non-decreasing and
    \begin{align*}
    Cov\left[f(I_i,i \in A),g(I_j,j \in B)|I_k,k \in C\right]&=Cov\left[\bar{f}(J_i,i \in A),\bar{g}(J_j,j \in B)|J_k,k \in C\right]\\
    &\leq 0
    \end{align*}
where the inequality uses the fact that  Till\'e's elimination procedure is CNA, bt Lemma \ref{lemma:Tille}.
\end{proof}

\subsection{Proof of Theorem \ref{theo3}} \label{app:brew}

 Brewer's method is presented in Algorithm \ref{bre:samp} as a particular case of Algorithm \ref{algo1}, where $T=n$ and where, for all $t$, $\delta(t)$ is such that
    \begin{eqnarray*}
      \delta_k(t) & = & \left\{\begin{array}{ll}
                                 0 & \textrm{if } k \in \{J_1,\ldots,J_{t-1}\}, \\
                                 1-\pi_{J_t}(t-1) & \textrm{if } k=J_t, \\
                                 - \frac{\pi_k(t-1)\{1-\pi_{J_t}(t-1)\}}{n-t+1-\pi_{J_t}(t-1)} & \textrm{otherwise.}
                               \end{array}
       \right.
    \end{eqnarray*}
This leads to
    \begin{eqnarray} \label{pbrew:1}
      \sum_{k \in U(t)} |\delta_k(t)| & = & \{1-\pi_{J_t}(t-1)\} + \sum_{k \in U(t) \setminus J_t} \frac{\pi_k(t-1)(1-\pi_{J_t}(t-1))}{n-t+1-\pi_{J_t}(t-1)} \nonumber \\
      & = & \{1-\pi_{J_t}(t-1)\} \bigg(1+\sum_{k \in U(t) \setminus J_t} \frac{\pi_k(t-1)}{n-t+1-\pi_{J_t}(t-1)} \bigg) \nonumber \\
      & = & 2 \times \{1-\pi_{J_t}(t-1)\},
    \end{eqnarray}
where the third line in \eqref{pbrew:1} follows from the identity
    \begin{eqnarray*}
      \sum_{k \in U(t)} \pi_k(t-1) & = & n-t+1.
    \end{eqnarray*}
This shows that
\begin{align*}
       Pr\Big( \sum_{k \in U(t)} |\delta_k(t)|  \leq 2\Big)=1,\quad t=1,\dots,n
\end{align*}
and the result follows from Theorem \ref{theo1}.

\end{document}